\documentclass[reqno]{amsart}

\usepackage{a4wide}
\usepackage{amsmath, amssymb, amsthm, mathtools, bbm}
\usepackage[outdir=./]{epstopdf}
\usepackage[left=3.2cm, right=3.2cm, bottom=3.2cm]{geometry}
\usepackage[british]{babel}
\usepackage[utf8]{inputenc}
\usepackage[T1]{fontenc}
\usepackage{hyperref}
\usepackage{csquotes}
\usepackage{tikz-cd}
\usepackage{enumitem}
\setenumerate[1]{label={\arabic*)}}

\usepackage{stmaryrd}

\usepackage[backend=biber, style=alphabetic, giveninits=true,maxnames=99,maxalphanames=99]{biblatex}
\usepackage{graphicx}

\newtheorem{theorem}{Theorem}[section]
\newtheorem{lemma}[theorem]{Lemma}
\newtheorem{proposition}[theorem]{Proposition}
\newtheorem{corollary}[theorem]{Corollary}

\theoremstyle{definition}
\newtheorem{definition}[theorem]{Definition}

\theoremstyle{remark}
\newtheorem{remark}[theorem]{Remark}
\newtheorem{example}[theorem]{Example}

\newtheorem{construction}[theorem]{Construction}

\newcommand{\Z}{\mathbb{Z}}
\newcommand{\N}{\mathbb{N}}

\newcommand{\CAlg}{\mathrm{CAlg}}
\newcommand{\Sp}{\mathrm{Sp}}
\newcommand{\dr}[3]{\Omega^{#3}_{#1/#2}}

\newcommand{\Alg}{\mathrm{Alg}}
\newcommand{\LMod}{\mathrm{LMod}}
\newcommand{\RMod}{\mathrm{RMod}}

\newcommand{\op}{\mathrm{op}}
\newcommand{\C}{\mathcal{C}}
\newcommand{\D}{\mathcal{D}}

\newcommand{\Nm}{\mathrm{Nm}}

\renewcommand{\S}{\mathbb{S}}

\newcommand{\id}{\mathrm{id}}

\newcommand{\gr}{\operatorname{gr}}

\newcommand{\Mod}{\mathrm{Mod}}

\newcommand{\Lan}{\mathrm{Lan}}

\newcommand{\R}{\mathbb{R}}

\newcommand{\F}{\mathcal{F}}

\newcommand{\TCart}{\mathrm{TCart}}
\newcommand{\CycSp}{\mathrm{CycSp}}

\newcommand{\W}{\mathbf{W}}
\newcommand{\THH}{\mathrm{THH}}
\newcommand{\Fr}{\mathrm{Fr}}

\newcommand{\Fin}{\mathrm{Fin}}
\newcommand{\Mack}{\mathrm{Mack}}
\newcommand{\TP}{\mathrm{TP}}
\newcommand{\TR}{\mathrm{TR}}
\newcommand{\cris}{\mathrm{cris}}
\renewcommand{\F}{\mathbb{F}}
\renewcommand{\R}{\mathcal{R}}
\newcommand{\incl}{\mathrm{incl}}
\newcommand{\loccit}{\textit{loc. cit.}}
\newcommand{\Span}{\mathrm{Span}}
\newcommand{\lax}{\mathrm{lax}}
\newcommand{\E}{\mathcal{E}}
\newcommand{\Ab}{\mathrm{Ab}}

\newcommand{\fgt}{\mathrm{fgt}}

\DeclareMathOperator{\Fun}{\mathrm{Fun}}

\newcommand{\lbar}[1]{\mkern 1.5mu\overline{\mkern-1.5mu#1\mkern-1.5mu}\mkern 1.5mu}

\author{Konrad Bals}

\title{The topological Cartier--Raynaud ring}

\begin{document}

	\begin{abstract}
		We prove that the $\infty$-category of $p$-typical topological Cartier modules, recently introduced by Antieau--Nikolaus, over some base $A$ is equivalent to the $\infty$-category of modules over a ring spectrum $\mathcal R_A$, which we call the topological Cartier--Raynaud ring. Our main result is an identification of the homotopy groups of $\mathcal R_A$. 
		In particular, for $A=W(k)$ the homotopy groups $\pi_*\mathcal R_{W(k)}$ recover the classical Cartier--Raynaud ring constructed by Illusie--Raynaud. Moreover, along the way we will describe the compact generator of $p$-typical topological Cartier modules and classifies all natural operations on homotopy groups of $p$-typical topological Cartier modules. 
	\end{abstract}

	\maketitle
	
	\section{Introduction}
	
	Given a scheme $S$ over a ring $k$. One goal of algebraic geometry is to associate to this situation a cohomology theory $H^*(S)$ and deduce geometric properties of $S$ from algebraic information on $H^*(S)$. In the last years there has been an amazing development of such cohomology theories by analyzing constructions coming from stable homotopy theory: 
	
	To the scheme $S$ one can associate its topological Hochschild homology $\THH(S)$ and variants of it. Define $\TP(S):=\THH(S)^{tS^1}$ as the Tate construction of the $S^1$-action on $\THH(S)$. This spectrum prominently features in the development of prismatic cohomology, in fact in some situations it actually computes prismatic cohomology: 
	
	\begin{theorem}[\cite{BMS19}]\label{TP_and_prism}
		Let $R$ be a quasi-syntomic $k$-algebra, then there exists a motivic filtration on topological periodic cyclic homology $\TP(R)$ with 
		\[\gr^n\TP(R)\simeq \widehat\Delta_{R}^{(1)}\{n\}[2n]\]
		where the right hand side is version of prismatic cohomology, Breuil--Kisin twisted and shifted into degree $2n$.
	\end{theorem}

	This connection between arithmetic cohomology theories and variants of topological Hochschild homology is not new. In \cite{HM97} Hesselholt--Madsen describe the genuine $S^1$-action on topological Hochschild homology and construct the spectrum $\TR(S):=\lim_{n}\THH(S)^{C_{p^n}}$\footnote{By the result of \cite{NS18} this can be completely expressed in terms of non-genuine information, however, that makes the formula not as nice.}, topological restriction homology. It still carries a non-genuine $S^1$-action and we have in some situations:
	
	\begin{theorem}[{\cite[Theorem C]{Hes96}}]\label{TR_WOmega}
		Let $S$ be a smooth scheme over $\F_p$, then 
		\[\pi_*\TR(S)\cong W\dr S{\F_p}*\]
		where $W\dr S{\F_p}*$ is the de Rham--Witt complex of $S$ relative $\F_p$ computing $H^*_{\cris}(S)$.
	\end{theorem}

	The de Rham--Witt complex has been introduced by Bloch--Deligne--Illusie in \cite{Blo77} and \cite{Ill79} in order to study cristalline cohomology from an algebraic perspective, giving an alternative perspective to the site theoretic definition in \cite{Ber74}.

	Thus, in particular $\TR(S)$ knows about cristalline cohomology of $S$. In fact, this implies a special case of Theorem \ref{TP_and_prism}: By the machinery developed in \cite{BMS19} and \cite{BS22} over the prism $(\TR(\F_p)\cong \Z_p,p)$ prismatic cohomology agrees with cristalline cohomology, and in general, $\TP(S)\simeq\TR(S)^{tS^1}$ with the motivic filtration being the decalage of the Tate filtration on $\TP(S)\simeq\TR(S)^{tS^1}$. Using the computation of $\TR(S)$ from above this refines to an equivalence in $D(\Z_p)$ of
	\[\gr^n\TP(R)\simeq W\dr S{\F_p}*\]
	as shown in \cite{AN18}[Theorem 6.24].
	
	The Verschiebung and Frobenius operators on the cristalline cohomology groups can already be seen on the de Rham--Witt complex.
	There are algebraic maps on the complex level
	\[V,F\colon W\Omega_{S/\F_p}^*\to W\Omega_{S/\F_p}^*\]
	satisfying certain relations together with the differential. Over a general $\F_p$-algebra $k$ in \cite{IR83} Illusie and Raynaud describe the interaction of this structure as a module action on the de Rham--Witt complex by an explicit non-commutative graded ring $R(k)$:
	
	\begin{definition}
		
		For a ring $k$, the Cartier--Raynaud ring is given as 
		\[ R(k):=W(k)\{v,f,d\}/I_R. \]
		The quotient of the free associative graded (not necessarily central) $W(k)$-algebra on generators $v$ and $f$ in degree 0 and $d$ in degree 1. The ideal $I_R$ is generated by the following relations:
		\begin{alignat*}{3}
		f v&=p & vxf&=V(x)\\
		d f&=pf  d&v  d&=p d v\\
		fx&=F(x)f\qquad&xv&=F(x)v\\
		f  d v&= d& d^2&=0\\
		dx&=xd
		\end{alignat*}
		for all $x\in W(k)$ and where $V$ and $F$ are the Witt vector Verschiebung and Frobenius on $W(k)$.
		
	\end{definition}

	On the de Rham--Witt complex $W\Omega^*_{-/k}$ the elements $v,f$ act as Verschiebung $V$ resp. Frobenius $F$ and $d$ gives the differential $W\Omega^*_{-/k}\to W\Omega^{+1}_{-/k}$. Modules over $R(k)$ are called Cartier--Witt complexes and the study of the de Rham--Witt complex becomes an analysis of Cartier--Witt complexes. By \cite{HM04} the homotopy groups $\pi_*\TR(S)$ acquire a Verschiebung, a Frobenius and a differential becoming a module over $R(k)$ and in fact the equivalence in Theorem \ref{TR_WOmega} refines to an equivalence of Cartier--Witt complexes.
	
	Turning a statement on the homotopy groups into a statement in higher algebra, in \cite{AN18} Antieau and Nikolaus describe the spectrum $\TR$ as an object in the algebraic $\infty$-category of topological Cartier modules $\TCart_p$, which behave like a highly structured version of the category of Cartier--Witt complexes. 
	
	We make this precise by constructing a spectrum-level analogue of the Cartier--Raynaud ring as suggested in \cite{AN18}[Remark 3.13]: For every algebra $A\in\Alg(\TCart_p)$ the forgetful functor $\LMod_A(\TCart_p)\to\Sp$ lifts to an equivalence 
	\[\LMod_A(\TCart_p)\simeq \LMod_{\R_A}(\Sp)\]
	for a ring spectrum $\R_A$. In fact, the ring spectrum is uniquely determined by this property.
	
	\begin{definition}[Definition \ref{definition_Cartier_Raynaud}]
		For $A\in\Alg(\TCart_p)$ we call the $\mathbb{E}_1$-ring spectrum $\R_A$ the topological Cartier--Raynaud ring.
	\end{definition}

	Our main result is to completely describe the homotopy groups of $\R_A$ in terms of the homotopy groups $\pi_*A$. We compute:
	
	\begin{theorem}[Theorem \ref{main_theorem}]
		For $A\in\Alg(\TCart_p)$ there is an isomorphism
		\[\pi_*\R_A\simeq (\pi_*A)\{v,f,d\}/I_{\R}\]
		where $|v|=|f|=0$ and $|d|=1$ and with the ideal $I_{\R}$ generated by the relations
		\begin{alignat*}{3}
		f v&=p & vxf&=V_A(x)\\
		d f&=pf  d&v  d&=p d v\\
		fx&=F_A(x)f\qquad&xv&=F_A(x)v\\
		f  d v&= d\ (+\eta\ \text{if}\ p=2)& d^2&=\eta  d\\
		dx&=d_A(x)+(-1)^{|x|}xd
		\end{alignat*}
		for all $x\in \pi_*A$ and where $V_A,F_A, d_A:\pi_*A\to \pi_*A$ are induced by the topological Cartier module structure of $A$.
	\end{theorem}

	For a ordinary ring $k$ we have $W(k)\in\Alg(\TCart_p)$ and putting $A=W(k)$ we get the justification for our naming:
	
	\begin{corollary}
		Let $k$ be an ordinary ring, then there is an isomorphism of associative graded rings
		\[\pi_*\mathcal R_{W(k)}\cong R(k)\]
		
	\end{corollary}

	By the construction of $\R_A$ in section \ref{Main}, the underlying spectrum carries the structure of a topological Cartier module itself. As such it is a compact generator of $\TCart_p$ and we get as a further consequence:
	
	\begin{corollary}[Corollary \ref{truncated}]
		The $\infty$-category $\LMod_A(\TCart_p)$ is compactly generator by a single generator and if $A$ is $n$-truncated then the compact generator is $(n+1)$-truncated.
	\end{corollary}

	Finally, the underlying spectrum of $\R_A$ as a compact generator of the $\infty$-category $\LMod_A\TCart_p$ corepresents the homotopy group functor $\pi_*\colon\LMod_A\TCart_p\to\gr\Ab$. Thus, $\pi_*\R_A$ captures all information about the ring of natural endotransformations of $\pi_*$.

	\begin{corollary}[Corollary \ref{homotopy_operations}]
	For the functor $\pi_*\colon\LMod_A\TCart_p\to\gr\Ab$ the graded ring of natural transformations $\pi_*\to\pi_*$ is given by $\pi_*\R_A$. In particular, all natural operations on the homotopy groups of $p$-typical topological Cartier modules are essentially given by Verschiebung, Frobenius, the differential and combinations thereof.  
	\end{corollary}

\subsection{Notation}
We will freely use the language of $\infty$-categories as developed in \cite{Lur09} and \cite{Lur16}. 

Moreover, we will work with the $\infty$-category $\Sp^{BS^1}:=\Fun(BS^1,\Sp)$ the $\infty$-category of spectra with $S^1$-action. The object $S^1$ carries an endomorphism given by multiplication by $p$, i.e $x\mapsto x^p$ and by abuse of notion we denote its induced map on classifying spaces by $p\colon BS^1\to BS^1$. Via restriction and Kan extension this gives us the adjunctions
\[\begin{tikzcd}
\Sp^{BS^1}\ar[r,"p^*"]&\Sp^{BS^1}\ar[l, bend right, shift right = 0, swap, "p_!"]\ar[l, bend left, shift left = 0, "p_*"]
\end{tikzcd}\]
classically referred to as $p_!=(-)_{hC_p}$ and $p_*=(-)^{hC_p}$, the homotopy orbits and homotopy fixed points with respect to the restricted $C_p$ action. However, this notation makes it clear, how to put a residual $S^1\cong S^1/C_p$-action on both of these adjoints.
For example, the residual action on $p_!\S[S^1]$ gives an equivalence $p_!\S[S^1]\simeq\S[S^1]$.

In the presence of group actions it is sometimes necessary to work with semilinear algebra, accounting for a twist by the group action. Nevertheless, for a ring $A$ we call a ring $R$ over $A$, an $A$-algebra, regardless of whether $A\to R$ factors through the $\mathbb{E}_1$-center of $R$.
In particular, for an ordinary ring $A$ we will write $A\{X_i\}$ for the free non-commutative, non-central $A$-algebra on generators $X_i$, i.e. $A\{X_i\}\cong A*\Z\{X_i\}$ the coproduct of associative algebras of $A$ and the group ring $\Z\{X_i\}$ of the free group on $\{X_i\}$.

\subsection{Acknowledgment}
I would like to thank Thomas Nikolaus for introducing me to topological Cartier modules and bringing up this topic. Moreover, I want to thank Achim Krause, Jonas McCandless and Shai Kedar for helpful discussions on this project, and Zhouhang \textsc{Mao} and Jonas McCandless for comments on an earlier draft. The project was funded by the
Deutsche Forschungsgemeinschaft (DFG, German Research Foundation) –
Project-ID 427320536 – SFB 1442, as well as under Germany's Excellence
Strategy EXC 2044 390685587, Mathematics Münster: Dynamics–Geometry–
Structure.

\section{The Category of p-typical Topological Cartier Modules}\label{Topological_Cartier_Modules}

We briefly recall the definitions of $p$-typical topological Cartier Module in \cite{AN18} and recall the perspective through Mackey functors made explicit in \cite{McC21}.

\begin{definition}[{\cite[Definition 3.1]{AN18}}]
	A topological Cartier module $X$ is a spectrum with $S^1$-action together with an $S^1$-equivariant factorization
	\[Nm_{C_p}\colon p_! X\xrightarrow V X\xrightarrow F p_*X\]
	of the $C_p$-Norm map.
\end{definition}

\begin{example}[\cite{AN18}[Construction 3.18]]
	For a cyclotomic spectrum $X$ the topological restriction homology $\TR(X)$ acquires the structure of a topological Cartier module.
\end{example}

By adjunction, a topological Cartier modules comes equipped with maps 
\begin{equation}\label{Verschiebung_Frobenius}
V\colon X\to p^*X,\qquad F\colon p^*X\to X.
\end{equation}
In particular as $X\simeq p^*X$ on underlying spectra, these maps give rise to operations $V,F\colon \pi_*X\to \pi_*X$. Moreover, the $S^1$ action on $X$ gives $\pi_*X$ the structure of a $\pi_*\S[S^1]\simeq \pi_*\S[d]/(d^2-\eta d)$-module for $|d|=1$, where $\eta$ is the unique non-zero element in $\pi_1\S$. The interaction of these operations has been studied in \cite{AN18} and they prove:

\begin{proposition}[{\cite[Lemma 3.33]{AN18}}]\label{V,F,d-structure_on_homotopy_groups}
	Given a topological Cartier module $X$, then the graded abelian group $\pi_*X$ carries a module action by the graded non-commutative ring
	\[(\Z\{v,f,d\}/I_\TCart\]
	where the action of $v$, $f$, and $d$ on $\pi_*X$ is given by $V$, $F$, and $d$, respectively, and $I_\TCart$ is generated by the relations
	
	\begin{alignat*}{3}
	f v&=p& d^2&=\eta  d\\
	d f&=pf  d&v  d&=p d v\\
	f  d v&= d\ (+\eta\ \text{if}\ p=2)
	\end{alignat*}
	
\end{proposition}
Moreover, the perspective in \eqref{Verschiebung_Frobenius} suggests, to view the Verschiebung and Frobenius as restriction and transfer along something like a $p$-fold covering map $S^1\to S^1$. This has been made precise in \cite{McC21}:

Following \cite{AMR17}, let $\W_p:=S^1\ltimes \N$ be the semidirect product of $\N$ acting on $S^1$ by $p$-power exponentiation, i.e. the map $\N\to \mathrm{End}(S^1)$ sends $n\in\N$ to the $p^n$-fold covering map $S^1\to (p^*)^nS^1\cong S^1$.
In \cite[Lemma 3.2.1]{McC21} McCandless verifies that the category $B\W_p$ is orbital in the sense of \cite{Bar21}, i.e. the finite coproduct completion $\Fin^\amalg B\W_p$ admits fibre products, such that the definition of Mackey functors on $B\W_p$ makes sense.

 By \cite[Proposition 5.5]{AN18} and \cite[Proposition 3.2.8]{McC21} we get an alternative description of p-typical topological Cartier modules.

\begin{proposition}
	There is an equivalence of categories 
	\[\TCart_p\simeq \Mack(B\W_p):=\Fun^\Pi(\Span(\Fin^\amalg B\W_p),\Sp)\]
\end{proposition}

Forgetting the Verschiebung maps the functor $B\W_p^\op$ to $\Span\Fin^\amalg B\W_p$ induces a forgetful functor
\[\TCart_p\to \CycSp_p^{\Fr}:=\Fun(B\W_p^\op,\Sp)\]
where the right hand side is given by $p$-typical cyclotomic spectra with Frobenius lift, cf. also the series of work \cite{AFR18,AMR17b,AMR17}. Restricting further to the point in $B\W_p$ give conservative functors 
\[\TCart_p\to \CycSp_p^{\Fr}\to \Sp\]
both of which preserve limits and colimits. Indeed, in $\CycSp_p^\Fr$ limits and colimits are computed underlying and, thus, the right functor even creates limits and colimits, and the composite also preserves limits and colimits because localizing onto product preserving functors (i.e. Mackeyfication) does not change the value on the point.

In particular, the functor $\TCart_p\to \CycSp_p^\Fr$ admits a left adjoint $(-)[V]\colon\CycSp_p^\Fr\to\TCart_p$, which has been explicitly described in \cite[Lemma 4.1]{AN18}: Let $M\in\CycSp_p^\Fr$ with Frobenius $M\to p_*M$ then
\begin{equation*}\label{V-induction}M[V]\simeq\bigoplus_{n\in\N}p_!^nM\end{equation*}
with Verschiebung given on summands by $V\colon p_!(p_!^nM)\xrightarrow{\sim}p_!^{n+1}M$ and similarly Frobenius given by the Frobenius on $M$ in the bottom summand $M\rightarrow p_*M$ and for $n>0$ given by $F\colon p_!^nM\xrightarrow{Nm} p_*p_!^{n-1}M$.

\section{Symmetric Monoidal Structures}

In this section we want to describe the natural symmetric monoidal structure on $\TCart_p$ and prove that with this structure the functor $(-)[V]\colon \CycSp_p^\Fr\to\TCart_p$ previously described refines to a symmetric monoidal functor.

Let us first recollect some abstract results on symmetric monoidal structures on functor categories.

\begin{definition}
	Let $\C$ and $\D$ be categories and given a symmetric monoidal structure on $\D$.
	\begin{enumerate}
		\item We denote by $\D^\C$ the pointwise symmetric monoidal structure on $\Fun(\C,\D)$ (cf. \cite[Remark 2.1.3.4.]{Lur16} ). It is characterized by 
		\begin{equation}\label{pointwise}
		\Fun^\lax(\E,\D^\C)\simeq\Fun(\C,\Fun^\lax(\E,\D))
		\end{equation}
		for every symmetric monoidal $\infty$-category $\E$.
		\item If also $\C$ has a symmetric monoidal structure, we denote by $\Fun(\C,\D)^\otimes$ the Day convolution product on $\Fun(\C,\D)$ (cf. \cite[Example 2.2.6.9.]{Lur16}) satisfying the universal property for every symmetric monoidal $\E$:
		\begin{equation}\label{Day}
		\Fun^\lax(\E,\Fun(\C,\D)^\otimes)\simeq\Fun^\lax(\E\times\C,\D).
		\end{equation}
	\end{enumerate}
\end{definition}

Both symmetric monoidal structures behave well under left Kan extension and we have:

\begin{lemma}\label{left_Kan_symmetric_monoidal}
	Given a functor $f\colon\C\to\C'$ and a symmetric monoidal $\infty$-category $\D$ such that the tensor product preserves colimits in each variable. Denote by $f_!\colon\Fun(\C,\D)\to\Fun(\C',\D)$ the left Kan extension.
	\begin{enumerate}
		\item Then $f_!$ refines to a symmetric monoidal functor $\D^\C\to\D^{\C'}$
		\item If we assume that $f$ is a symmetric monoidal functor between symmetric monoidal categories, then $f_!$ can also be refined to a symmetric monoidal functor \[\Fun(\C,\D)^\otimes\to\Fun(\C',\D)^\otimes\]
	\end{enumerate}
\end{lemma}
\begin{proof} By the universal property the lax structure on the functor $\D^\C\to\D^{\C'}$ can be obtained via left Kan extension $\C\to\Fun^\lax(\D^\C,\D)$ given by evaluation. 
		
The second claim is \cite[Theorem 3.6.]{MS21}. 
\end{proof}

In general the pointwise tensor product (forgetting possible symmetric monoidal structure on the source) and the Day convolution give different symmetric monoidal structures on the functor category. In some situations they agree. 

\begin{lemma}\label{cocartesian_Day_and_pointwise}
	Given two symmetric monoidal categories $\C$ and $\D$ and assume that $\C$ is cocartesian and the tensor product in $\D$ preserves colimits in each variable. Then the Day convolution and the pointwise symmetric monoidal structures on $\Fun(\C,\D)$ are naturally equivalent.
\end{lemma}
\begin{proof}
	We use the universal properties of the Day convolution and the pointwise structure and we get:
	\begin{align*}
	\Fun^\lax(-,\Fun(\C,\D)^\otimes)\overset{\eqref{Day}}{\simeq}&\Fun^\lax(-\times \C,\D)\overset{\eqref{Day}}{\simeq}\Fun^\lax(\C,\Fun(-,\D)^\otimes)\\
	\simeq&\Fun(\C,\Fun^\lax(-,\D))\overset{\eqref{pointwise}}{\simeq}\Fun^\lax(-,\D^\C)
	\end{align*}
	where the first equivalence in the second line uses \cite[Theorem 2.4.3.18]{Lur16} stating that for a cocartesian operad $\C$ the map
	\[\Fun^\lax(\C,\E)\to \Fun(\CAlg(\C),\CAlg(\E))\simeq \Fun(\C,\CAlg(\E))\]
	is an equivalence.
\end{proof}

We can now describe the symmetric monoidal structure on our categories of interest:

\begin{construction}
	In the case of 
	\[\TCart_p\simeq \Mack(\Fin^\amalg(B\W_p))\] 
	the general theory of symmetric monoidal structures on spectral Mackey functors developed in \cite{BGS20} equips $p$-typical topological Cartier modules with a symmetric monoidal structure (cf. 3.8. in \loccit): 
	
	The category $\Fin^\amalg BW$ admits products by the cited result from \cite{McC21}, which inherit a symmetric monoidal structure to $\Span\Fin^\amalg(BW)$ making the inclusion $\Fin^\amalg(BW)\to\Span\Fin^\amalg(BW)$ symmetric monoidal. Then Day convolution on the functor category $\Fun(\Span\Fin^\amalg(BW),\Sp)$ localizes to the symmetric monoidal structure on Mackey functors as constructed in \cite{BGS20}. 
	We will denote this tensor product by
	\[-\boxtimes -:\TCart_p\times\TCart_p\to\TCart_p\]
	
	On the $\infty$-category $\CycSp_p^\Fr\simeq \Fun(B\W_p^\op,\Sp)$ we put the pointwise tensor product $\Sp^{B\W_p^\op}$.
\end{construction}

\begin{theorem}\label{symmetric_monoidal_structre}
	The functor $(-)[V]\colon\CycSp_p^\Fr\to\TCart_p$ constructed in \cite{AN18} refines to a symmetric monoidal functor with respect to the symmetric monoidal structures described above.
\end{theorem}
\begin{proof}
	The functor $(-)[V]\colon\CycSp_p^\Fr\to\TCart_p$ is left adjoint to the forgetful functor coming from the restriction $B\W^\op\to\Span(\Fin^\amalg B\W_p)$. In particular, this left adjoint can be written as the following factorization
	\begin{align*}
	\CycSp_p^\Fr\simeq \Fun(B\W_p^\op,\Sp)\xrightarrow{\Lan}&\Fun(\Fin^\amalg(B\W_p)^\op,\Sp)\xrightarrow{\Lan}\Fun(\Span\Fin^\amalg(B\W_p),\Sp)\\
	\xrightarrow M&\Mack(B\W_p)\simeq\TCart_p
	\end{align*}
	where $\Lan$ are the respective left Kan extension and $M$ is the localization onto Mackey functors. The latter carries a symmetric monoidal structure by \cite[Lemma 3.7.]{BGS20}. Note that the product on $\Fin^\amalg(B\W)$ gives the opposite category $\Fin^\amalg(B\W)^\op$ the cocartesian symmetric monoidal structure, such that by Lemma \ref{cocartesian_Day_and_pointwise} the pointwise and Day convolution product on $\Fun(\Fin^\amalg(B\W)^\op,\Sp)$ agree. Thus the left Kan extensions lift to symmetric monoidal functors by Lemma \ref{left_Kan_symmetric_monoidal} 1) in the first case and 2) in the second case.
\end{proof}

\begin{corollary}\label{lax_monoidal}
	The symmetric monoidal structure on $\TCart_p$ from Theorem \ref{symmetric_monoidal_structre} induces a lax symmetric monoidal structure on the forgetful functor $\TCart_p\to\Sp$.
\end{corollary}

\begin{proof}
	The forgetful functor $\CycSp^\Fr_p\to\Sp$ is symmetric monoidal, so it satisfies to argue that $\TCart_p\to\CycSp^\Fr_p$ is lax symmetric monoidal, but this comes from the symmetric monoidal structure of its left adjoint.
\end{proof}

In \cite[Section 4]{AN18} many results can be directly imported to a non-completed tensor product on topological Cartier modules. In particular, we can recover the analogue of Corollary 4.4 in \loccit:

\begin{corollary}
	Given $M_1,M_2\in\TCart_p$ the tensor product $M_1\boxtimes M_2$ is equivalent to the total cofibre of a square
	\[\begin{tikzcd}
	\big(p_!M_1\otimes p_!M_2\big)[V]\ar[r]\ar[d]&
	\big(p_!M_1\otimes M_2\big)[V]\ar[d]\\
	\big(M_1\otimes p_!M_2\big)[V]\ar[r]& (M_1\otimes M_2)[V]
	\end{tikzcd}\]
	where the maps are induced by the map $p_!M_i\xrightarrow{V-\incl} M_i[V]\simeq \bigoplus_n
	p_!^nM_i$ of cyclotomic spectra with Frobenius lifts with the former being equipped with the 0-Frobenius lift.
\end{corollary}
\begin{proof}
	This is basically \cite[Corollary 4.4.]{AN18}. It follows from the cofibre sequence 
	\begin{equation}\label{totcof-tensor-formula}
	p_!M_i[V]\xrightarrow{V-\incl}M_i[V]\to M_i
	\end{equation}
	together with the fact that the tensor product $(-)\boxtimes(-)$ preserves colimits in both arguments separately and that $(-)[V]$ is symmetric monoidal by Theorem \ref{symmetric_monoidal_structre}.
\end{proof}


We will not use this explicit description of the tensor product of topological Cartier modules, but let us add the following remark for completeness reasons:
\begin{remark}
	One can explicitly identify the maps in \eqref{totcof-tensor-formula}, and this gives us an explicit description of bilinear maps in $\TCart_p$. Given $M_1,M_2, N\in\TCart_p$ a map $M_1\boxtimes M_2\to N$ is given by the following datum:
	\begin{enumerate}
		\item A morphism $f\colon M_1\otimes M_2\to N$ of underlying cyclotomic spectra with Frobenius lift\\
		\item Homotopies filling the diagrams
		\[\begin{tikzcd}
		M_1\otimes p_!M_2\ar[r,"{\id\otimes V}"]\ar[d,"{F\otimes\id}"]&M_1\otimes M_2\ar[dddd,"f"]\ar[ldddd,Rightarrow, shorten <= 2em, shorten >= 4em, shift left = 4,"H_1"]\\
		p_*M_1\otimes p_!M_2\ar[d, "\sim" {rotate = 90, anchor = south}]\\
		p_!(p^*p_*M_1\otimes M_2)\ar[d,"{p_!(\eta\otimes id)}"]\\
		p_!(M_1\otimes M_2)\ar[d,"{p_!f}"]\\
		p_!N\ar[r,"V"]&N
		\end{tikzcd}
		\begin{tikzcd}
		p_!M_1\otimes M_2\ar[r,"{V\otimes \id}"]\ar[d,"{\id\otimes F}"]&M_1\otimes M_2\ar[dddd,"f"]\ar[ldddd,Rightarrow, shorten <= 2em, shorten >= 4em, shift left = 4, "H_2"]\\
		p_!M_1\otimes p_*M_2\ar[d, "\sim" {rotate = 90, anchor = south}]\\
		p_!(M_1\otimes p^*p_*M_2)\ar[d,"{p_!(id\otimes \eta)}"]\\
		p_!(M_1\otimes M_2)\ar[d,"{p_!f}"]\\
		p_!N\ar[r,"V"]&N
		\end{tikzcd}\]
		\item A 2-cell witnessing the equivalence
		\[H_1\circ (V\otimes \id)\simeq H_2\circ (\id\otimes V)\]
		as homotopies between functors $p_!M_1\otimes p_!M_2\to N$.
	\end{enumerate}
\end{remark}

\section{The compact Generator of \texorpdfstring{$\TCart_p$}{TCartp}}\label{Main}

In section \ref{Topological_Cartier_Modules} we have already encountered the forgetful functor $\TCart_p\to\Sp$ and seen that it preserves limits, colimits and is conservative. Let $\R:\Sp\to\TCart_p$ denote a left adjoint. The object $\R(\S)$ is a compact generator of $\TCart_p$, cf. \cite{AN18}[Remark 3.13]. More generally, if $A\in\Alg(\TCart_p)$ is an algebra object, the $\infty$-category $\LMod_A(\TCart_p)$ is compactly generated by $\R(\S)\boxtimes A$.

\begin{theorem}\label{additive_Cartier--Raynaud}
	The underlying spectrum of the compact generator $\R(\S)$ of $\TCart_p$ is given by 
	\[	\R(\S)\simeq\bigoplus_{i\in\Z}(\S[V]\oplus\Sigma\S[V]).\] 
	Moreover, if $A$ is an algebra in $\TCart$, we have the functorial description $\R(\S)\boxtimes A\simeq\bigoplus_{i\in\Z}(A\oplus \Sigma A)$ of the compact generator of $\LMod_A(\TCart_p)$.  
	In particular, if $A$ is $n$-truncated $\R(\S)\boxtimes A$ is $(n+1)$-truncated.  
\end{theorem}

\begin{proof}
	The statement about the compact generator of $\LMod_A(\TCart_p)$ follows immediately from the statement for $A=\S[V]$, the unit in $\TCart_p$. We can factor the left adjoint $\R$ into the two left adjoints
	\[\Sp\xrightarrow{-\otimes\S[\W_p]}\CycSp_p^\Fr\xrightarrow{(-)[V]}\TCart_p\]
	as they are all adjoint to the forgetful functor. The first functor is, indeed, as stated because on $\CycSp_p^\Fr\simeq\Fun(B\W_p^\op,\Sp)$ the left adjoint to the forgetful functor is given by tensoring with the induction on $\S$. In \cite[Example 2.2.9.]{McC21} the object $\S[\W_p]$ has already been computed, and we have
	\[\S[\W_p]\simeq\bigoplus_{m\geq 0}\S[S^1/C_{p^m}]\simeq\bigoplus_{m\geq 0}(p^*)^m\S[S^1]\]
	with Frobenius lifts induced summandwise by the counit of the $(p^*\dashv p_*)$-adjunction, explicitly $(p^*)^m \S[S^1]\to p_*(p^*)^{m+1}\S[S^1]$.
	Thus, by expanding and reordering the terms using the projection formula $(S^1/C_p)^{hC_p}\simeq S^1\times BC_p$ we can compute
	\begin{align*}
	\R(\S)&\simeq\S[\W_p][V]\simeq\bigoplus_{n,m\geq 0}p_!^n(p^*)^m\S[S^1]\\
	&\simeq\left(\bigoplus_{n\geq m\geq 0}p_!^{n-m}\S[S^1]\otimes p_!^m\S\right)\oplus\left(\bigoplus_{0\leq n<m}(p^*)^{m-n}\S[S^1]\otimes p_!^n\S\right)\\
	&\simeq\left(\bigoplus_{i=n-m\geq0}p_!^i\S[S^1]\oplus\bigoplus_{j=m-n>0}(p^*)^{j}\S[S^1]\right)\otimes\bigoplus_{k\geq 0}p_!^k\S\simeq \lbar\R\otimes \S[V]
	\end{align*}
	where we set $\lbar\R:=\bigoplus_{i\geq 0}p_!^i\S[S^1]\oplus\bigoplus_{j>0}(p^*)^j\S[S^1]$. But here the underlying spectrum of $ p_!\S[S^1]$ and $p^*\S[S^1]$ is given by $\S\oplus\Sigma \S$. Thus, $\lbar\R\simeq \bigoplus_{i\in\Z}\S\oplus\Sigma\S$ and the claim follows. 
\end{proof}	

	Because $\R(\S)\boxtimes A$ only consists of $A$ and $\Sigma A$ as summands, we can immediately deduce the following statement on truncatedness:

	\begin{corollary}\label{truncated}
		If $A\in\Alg(\TCart_p)$ is $n$-truncated, then the compact generator of $\LMod_A(\TCart_p)$ is $(n+1)$-truncated.		
	\end{corollary}
	
	The object $\R(\S)$ is a commutative algebra in $\TCart_p$ with respect to $\boxtimes$, however, it is not the unit. Moreover, the underlying spectrum of $\R(\S)$ carries another $\mathbb{E}_1$-structure. It is this $\mathbb{E}_1$-structure that we want to describe in this paper.

	\begin{definition}\label{definition_Cartier_Raynaud}
		For an algebra $A\in\Alg(\TCart_p)$ we define the topological Cartier--Raynaud ring over $A$ as
		\[\R_A:=\mathrm{end}_{\LMod_A(\TCart_p)}(\R(\S)\boxtimes A)^\op,\]
		i.e. the endomorphism spectrum of $\R(\S)\boxtimes A$ in $\LMod_A(\TCart_p)$ acquires an $\mathbb{E}_1$-algebra structure by composition and we equip $\R_A$ with the opposite ring structure. 
	\end{definition}

	\begin{remark}
		The reason for the $(-)^\op$ lies in the fact, that given an algebra $A$ in spectra, $A$-linear endomorphisms of $A$ as a left $A$-module naturally acquire a right $A$-module structure. In particular, the identification as rings naturally has the form 
		\[\mathrm{end}_{\LMod_A(\Sp)}(A)^\op\simeq A.\]
	\end{remark}

\begin{proposition}\label{Tcart_as_modules}
	For $A\in\CAlg(\TCart_p)$ there is an equivalence
	\[\LMod_A(\TCart_p)\simeq \LMod_{\R_A}(\Sp)\]
\end{proposition}
\begin{proof}
	The $\infty$-category $\LMod_A(\TCart_p)$ is compactly generated by $\R(\S)\boxtimes A$ and the stated equivalence follows from the Schwede-Shipley theorem proved in this language in \cite[Theorem 7.1.2.1.]{Lur16}. 
\end{proof}

In fact, this equivalence does not affect the underlying spectrum, thus, $\R_A$ controls the homotopy group functor $\pi_*\colon\TCart_p\to\gr\Ab$ and we get the following corollary from the Yoneda lemma:

\begin{corollary}\label{homotopy_operations}
	For $A\in\CAlg(\TCart_p)$ the graded ring of natural endomorphisms of the homotopy groups functor $\pi_*\colon\LMod_A(\TCart_p)\to\gr\Ab$ is given by $\pi_*\R_A$.
\end{corollary}
\begin{proof}
We can identify $\TCart_p\simeq \LMod_{\R_A}(\Sp)$ and let $h\LMod_{\R_A}(\Sp)$ denote its $\Ab$-enriched homotopy category. Now $\Sigma^k\R_A$ as left modules over $\R_A$ corespresent the individual homotopy group functors $\pi_k$ in $h\LMod_{\R_A}(\Sp)$. Thus, via (co)Yoneada homotopy operations $\pi_k\to \pi_{k+n}$ correspond to homotopy classes in $[\Sigma^{k+n}\R_A,\Sigma^k\R_A]\cong\pi_n\mathrm{end}_{\LMod_{\R_A}(\Sp)}(\R_A)$. Moreover, we can contravariantly identify the graded ring of natural endomorphisms of $\pi_*$ with the graded endomorphism ring $\pi_*\mathrm{end}_{\LMod_{\R_A}(Sp)}(\R_A)$. Finally, the isomorphism $\mathrm{end}_{\LMod_{\R_A}(\Sp)}(\R_A)\simeq \R_A^\op$ cancels out with the contravariance and we get the claimed identifications of graded rings. 
\end{proof}

Because $\R(\S)\boxtimes A$ is the image of $\S$ under the left adjoint to the forgetful functor $\Mod_A(\TCart_p)\to\Sp$, the underlying spectrum of $\R_A$ agrees with the underlying spectrum of $\R(\S)\boxtimes A$. 
By the Theorem \ref{additive_Cartier--Raynaud} we already have an additive description of the homotopy groups of $\R_A$ in terms of $\pi_*A$ given as
\[\pi_*\R_A\cong\bigoplus_{i\in\Z}\left(\pi_*A\oplus \pi_{*-1}A\right).\]
The goal of the remainder of this section is to compute the graded associative ring structure on $\pi_*\R_A$. For this we need the following lemma, which proves itself exactly like Corollary \ref{homotopy_operations} before.

\begin{lemma}\label{Algebra_Map_Through_Yoneda}
	Given $R\in\Alg(\Sp)$ an associative ring spectrum and $S_*$ an arbitrary graded ring in $\Ab$ such that the functor $\pi_*\colon\LMod_R(\Sp)\to\gr\Ab$ factors through $\LMod_{S_*}(\gr\Ab)$, then there is a graded ring map $S_*\to\pi_*R$ making the diagram commutative:
	\[\begin{tikzcd}
	\LMod_R(\Sp)\ar[r]\ar[d,"\pi_*"]&\LMod_{S_*}(\gr\Ab)\ar[d,"\fgt"]\\
	\LMod_{\pi_*R}(\gr\Ab)\ar[r,"\fgt"]\ar[ru,dashed]&\gr\Ab
	\end{tikzcd}\]
\end{lemma}
\begin{proof}
	We will construct the map $S_*\to\pi_*R$ as a map of ordinary graded rings elementwise. Every homogenous element $a\in R_n$ gives rise to a natural transformation $\pi_0\to\pi_{n}$ as functors $\LMod_R(\Sp)\to\gr\Ab$ via the natural $S_*$-module structure on $\pi_*$. 
	Denote by $h\LMod_R(\Sp)$ the $\Ab$-enriched homotopy category of $\LMod_R(\Sp)$. Then for every $k$ the functor $\pi_k$ factors through $h\LMod_R(\Sp)$ and is corepresented there by $\Sigma^k R$ as a left $R$-module. Thus, via the (co)Yoneda lemma this natural transformation corresponds to a homotopy class in $[\Sigma^n R,R]\cong \pi_n\mathrm{end}_{\LMod_R(\Sp)}(R)$. 
	The multiplicative structures are given by composition and because of the contravariance of coYonda we get a ring map $S_*\to\pi_*\mathrm{end}_{\LMod_R(Sp)}(R)^\op$. Now the identification $\mathrm{end}_{\LMod_R(\Sp)}(R)^\op\simeq R$ finishes the proof.
\end{proof}

\begin{example}\label{twisted group algebra}
	For $R\in\Alg(\Sp)$ and a left $R$-module spectrum $M$ with (left) $R$-linear $S^1$-action, i.e. $M\in\left(\LMod_R(\Sp)\right)^{BS^1}$, the $S^1$-action induces (left) $R$-linear maps $\Sigma M\simeq S^1\otimes M\to M$. In particular $\pi_*M$ is a left module over $\pi_*(R)[d]$ for $|d|=1$. By Morita theory $(\RMod_R\Sp)^{BS^1}\simeq \RMod_{R[S^1]}\Sp$ we get a map $\pi_*(R)[d]\to\pi_*(R[S^1])$. In fact one can show that the map is surjective with kernel generated by $d^2-\eta d$.

\end{example}

We will need a variant of this example. The $\infty$-category $\Sp^{BS^1}$ has a symmetric monoidal structure coming from the underlying symmetric monoidal structure of $\Sp$ equipped with the diagonal action. 
In particular, we can consider $R\in\Alg(\Sp^{BS^1})$ and look at $M\in\LMod_R(\Sp^{BS^1})$, i.e. modules with $S^1$-equivariant module structure with respect to the non-trivial action on $R$.
The induced maps $\Sigma M\to M$ are not $R$-linear anymore, but twisted by the $S^1$-action on $R$, thus, the underlying spectrum of $M$ will only be a module over a twisted version of $R[S^1]$.

By abstract nonsense the $\infty$-category $\LMod_R(\Sp^{BS^1})$ is compactly generated by the module $R\otimes \S[S^1]$ equipped with the diagonal $S^1$-action. Let's write $R^\tau[S^1]$ for this object, it is sometimes called the twisted group algebra.
 
\begin{proposition}
	Let $R$ be an $S^1$-equivariant ring spectrum. With the notation from above we have 
	\[\pi_*(R^\tau[S^1])\cong \pi_*R\{d\}/(d^2=\eta d, dr=d_Rr+(-1)^{|r|}rd,\forall r\in\pi_*R)\] where $d_R$ is the image of $d$ under the induced unit map $\pi_*\S[S^1]\to \pi_*R$.B
\end{proposition}

\begin{proof}
	The unit map $\S\to R$ on underlying ring spectra, extends to an algebra map $\S[S^1]\to R$. Moreover, due to the diagonal action on $R\otimes \S[S^1]$, this extension of the underlying unit map $\S\to R\otimes\S[S^1]$ factors as
	\[\S[S^1]\xrightarrow{\Delta}\S[S^1]\otimes\S[S^1]\rightarrow R\otimes\S[S^1]\]
	But now for $d\in\pi_*\S[S^1]$, we have $\Delta_*(d)=d\otimes 1+1\otimes d$, so that $d$ is sent to $(d_R\otimes 1+1\otimes d)$ under this map. Thus, the action of $d$ on $\pi_*(R^\tau[S^1])$ is given by multiplication with this element, giving the claim.
\end{proof}

In Proposition \ref{V,F,d-structure_on_homotopy_groups} we have recalled the structure on the homotopy groups of topological Cartier modules as described in \cite{AN18}. In particular, we have an example of the above Lemma \ref{Algebra_Map_Through_Yoneda}: 

\begin{example}\label{Example_Ring_Map}
	The homotopy groups functor $\pi_*\colon\TCart_p\to\gr\Ab$ factors through $\Mod_{S_*}\gr\Ab$ for $S_*:=\Z[v,f,d]/I_{\TCart_p}$ by Proposition \ref{V,F,d-structure_on_homotopy_groups}. Thus, we get a map $\Z[v,f,d]/I_{\TCart_p}\to \R_{\S[V]}$. Moreover, if $A\in\Alg(\TCart_p)$ is an algebra, then for $M\in\Mod_A(\TCart_p)$ the underlying spectrum of $M$ also carries a natural $A$-algebra structure by Corollary \ref{lax_monoidal}, and we get a map $\pi_*A\to\R_A$.
	
	Putting this together, we get a map from the coproduct of associative algebras, explicitly the map
	\begin{equation}\label{crucial map}\pi_*A\{v,f,d\}/I_{\TCart_p}\to \pi_*\R_A\end{equation}
	where the left hand side is the free graded associative, non-central $\pi_*A$ algebra on the generators $v, f$ and $d$.
\end{example}

In order to understand the map \eqref{crucial map} by the proof of Lemma \ref{Algebra_Map_Through_Yoneda} we have to understand the operations $V, F$ and $d$ on $\pi_*(\R(\S)\boxtimes A)$. This will be easier if we assume $A=A_0[V]$ for some $A_0\in\Alg(\CycSp_p^\Fr)$, as we have easier description of the Verschiebung und Frobenius in this case. Before we dive into the explicit identifications, let's recall the cyclotomic or topological Cartier module structures on important players:
\begin{equation}\label{free_Frobenius}
	\S[\W_p]\simeq \bigoplus_{m\geq 0}(p^*)^m\S[S^1],\qquad F\colon p^*(p^*)^m\S[S^1]\xrightarrow[\sim]{\id}(p^*)^{m+1}\S[S^1]
\end{equation}
and for $M\in\CycSp^\Fr_p$ we have
\begin{align}
	&M[V]\simeq \bigoplus_{n\geq 0}p_!^nM\nonumber\\
	&V\colon p_!^nM\xrightarrow{\eta} p^*p_!^{n+1}M,\qquad F\colon
	\begin{cases}
		p^*M\to M& n=0\\
		p^*p_!^{n}M\xrightarrow{p^*\Nm}p^*p_*p_!^{n-1}M\xrightarrow{\varepsilon} p_!^{n-1}M& n>0
	\end{cases}	\label{free_Verschiebung}
\end{align}
where $\eta, \varepsilon$ are the unit and counit of their respective adjunctions.

\begin{lemma}\label{generators}
	Let $A_0\in\Alg(\CycSp_p^\Fr)$ and set $A:=A_0[V]$. Then $A$ is an algebra in $\TCart_p$ and the map \eqref{crucial map} from above $(\pi_*A)\{v,f,d\}/I_{\TCart_p}\to \pi_*\R_A$ is surjective.
\end{lemma}
\begin{proof}
	Recall from the proof of Lemma \ref{Algebra_Map_Through_Yoneda} the construction of this map. The elements $v, f$ and $d$ are send to $\pi_*\R(A)\cong \pi_*\mathrm{end}(\R(\S)\boxtimes A)$ correspoding to the endomorphism of $\R(\S)\boxtimes A$ corepresenting $V, F$ and $d$ on the homotopy groups functor $\LMod_A(\TCart_p)\to\gr\Ab$.
	We will use that we can completely describe the $V, F$ and $d$ operators on $\pi_*(\R(\S)\boxtimes A)$. For this let us denote by $v, f$ resp. $d$ the image of the mulitplicative unit $1\in\pi_0\R(\S)\boxtimes A$ under $V, F$ resp. $d$. In other words, $v, f$ and $d$ are the image of their counterparts from $(\pi_*A)\{v,f,d\}$ under the map above.

	By our computation in Theorem \ref{additive_Cartier--Raynaud} and more precisely by the proof of it, we already have the additive identifications
	\[\R(\S)\boxtimes A\simeq \bigoplus_{n,m\geq 0}p_!^n\left((p^*)^m\S[S^1]\otimes A_0\right)\simeq \bigoplus_{i\geq 0} \Big(p_!^i\S[S^1]\otimes A\Big)\oplus\bigoplus_{j>0} \Big((p^*)^j\S[S^1]\otimes A\Big)\]
	where the $S^1$-action on the tensor products is given diagonally. In particular, as $p_!\S[S^1]\simeq \S[S^1]$ in $\Sp^{BS^1}$, we have $p_!^i\S[S^1]\otimes A\simeq A^\tau[S^1]$. Nevertheless, we will continue writing the term on the left hand side in order to be able to distinguish different summands.
	
	\textit{Claim 1:} The summands $\pi_*(p_!^i\S[S^1]\otimes A)$ are free as right $\pi_*(A^\tau[S^1])$-modules on $v^i$ for all $i\geq 0$. Furthermore, $V$ maps $A$ in $p_!^i\S[S^1]\otimes A$ isomorphically to $A$ in $p_!^{i+1}\S[S^1]\otimes A$.
	
	The first statement follows from the second because the $i$-fold Verschiebung of 1 is $v^i$. So we want to have a close look at the Verschiebung $\R(\S)\boxtimes A\to p^*(\R(\S)\boxtimes A)$. Because $\R(\S)\boxtimes A$ by assumption on $A$ lies in the image of $-[V]\colon\CycSp_p^\Fr\to\TCart_p$ we can apply the description from \eqref{free_Verschiebung}, i.e. the Verschiebung is given by the unit $\id\to p^*p_!$ and we get
	\[p_!^i\S[S^1]\otimes A\to p^*p_!^{i+1}(\S[S^1]\otimes  A)\simeq p^*p_!^{i+1}\S[S^1]\otimes p^*A\hookrightarrow p^*(\R(\S)\boxtimes A),\]
	not affecting the $A$-factor. Moreover, indeed, the map $p_!^i\S[S^1]\to p^*p_!^{i+1}\S[S^1]$ is given by the identity on the $\S$-summand. 
	
	\textit{Claim 2:} The two summands $\pi_*( A)$ resp. $\pi_*( A[1])$ of $\pi_* ((p^*)^j\S[S^1]\otimes A)$ are free as left $\pi_*( A)$-modules on $f^j$ resp. $f^j d$ for $j>0$. Moreover, the Frobenius operator
	\[\pi_*( A)\cdot 1\to \pi_*( A)\cdot f,\qquad \pi_*( A)\cdot d\to \pi_*( A)\cdot f d\]
	is given by the Frobenius of $\pi_*(A)$. That is, for $x\in\pi_*( A)\subset \pi_*(\R(\S)\boxtimes A)$ we have the relation $fx=F_{ A}(x)f$, where $F_A\colon \pi_*A\to\pi_*A$ is the Frobenius on $A$.
	
	Before we prove the claim, note that $(p^*)^j\S[S^1]\otimes  A$ is not free as a left nor right $A^\tau[S^1]$-module because the $S^1$-actions are different, it is sped up by a factor of $p^j$ on the former. If we have proven the claim, this yields another manifestation of the known identity $df^j=p^jf d$.
	
	In the claim, again the first statement follows from the last one because the Frobenius $F_A$ on $\pi_*( A)$ as a ring map sends 1 to 1. Thus, we are reduced to understanding the Frobenius $p^*(\R(\S)\boxtimes A)\to\R(\S)\boxtimes A$ on the summand $(p^*)^j\S[S^1]\otimes  A$. Splitting up $A\simeq A_0[V]\simeq\bigoplus p_!^n A_0$ into further summands we will look at the Frobenius for $n=0$ and $n>0$ separately, as described in \eqref{free_Verschiebung}. In the first case it is given by the Frobenius on the tensor product
	\[p^*((p^*)^j\S[S^1]\otimes A_0)\hookrightarrow \R(\S)\boxtimes A,\]
	which is the tensor product of the Frobenius described in \eqref{free_Frobenius}, i.e. the identiy, and the Frobenius on $A_0$. And for $n>0$, we have the Frobenius given as
	\[
	\begin{tikzcd}
	p^*p_!^n((p^*)^{n+j}\S[S^1]\otimes A_0) \ar[r,"p^*\Nm"]\ar[d,phantom,sloped,"\simeq"]& p^*p_*p_!^{n-1}((p^*)^{n+j}\S[S^1]\otimes A_0)\ar[r]\ar[d,phantom,sloped,"\simeq"]& p_!^{n-1}((p^*)^{n+j}\S[S^1]\otimes A_0)\ar[d,phantom,sloped,"\simeq"]\\
	p^*(p^*)^j\S[S^1]\otimes p^*p_!^nA_0\ar[r,"{\id\otimes p^*\Nm}"]&(p^*)^{j+1}\S[S^1]\otimes p^*p_*p_!^{n-1}A_0\ar[r]&(p^*)^{j+1}\S[S^1]\otimes p_!^{n-1}A_0.
	\end{tikzcd}
	\]
	It is precisely the Frobenius of $A_0[V]$ as described in \eqref{free_Frobenius} tensored with the identity on $(p^*)^{j+1}\S[S^1]$.
	
	Putting both claims together, we see that $v^i, dv^i, f^j, f^jd$ give a full list of generators of $\pi_*(\R(\S)\otimes A)$ as a $\pi_*( A)$ module, and they all lie in the image of our given map.
\end{proof}

We already know that the kernel of the map $(\pi_*A)\{v,f,d\}\to\pi_*\R_A$ contains the ideal $I_{\TCart}$. But for example, we have also already seen, that in $\pi_*\R_A$ we find relations like $fx=F_A(x)f$ for $x\in\pi_*A$. We want to establish all relations of this sort:

\begin{lemma}\label{relations}
	In $\pi_*\R_A$ we have the following relations for all $x\in\pi_*A$:
	\[
	dx=d_A(x)+xd,\qquad fx=F_{ A}(x)f\qquad xv=vF_{ A}(x)\qquad vxf=V_{ A}(x)\]
\end{lemma}
\begin{proof}
	For the first equation recall the discussion from example \ref{twisted group algebra}. The object $\R_A$ is an algebra over $A^\tau[S^1]$ sending $d$ to $d$, thus, forcing the claimed relation in $\pi_*\R_A$.
	 
	The second equation, we have already seen in the proof of Lemma \ref{generators}. From this proof we have also already seen that left multiplication by $v$ induces an equivalence $\pi_*A\to v\cdot \pi_*A$, of right $\pi_*A$-modules. However, $v\cdot \pi_*A$ is a $\pi_*A$-bimodule, so $xv\in v\cdot \pi_*( A)$ has to be of the form $vx'$ for some $x'\in\pi_*A$. In $I_{\TCart}$ we have the relation $fdv=d$ (only modulo $\eta$ for $p=2$, but we can ignore this for the proof), so that $fdv x'=d x'$, and we can read of $x'$. On the other hand we have by our previously proven relations
	\begin{align*}
	fd xv=& fd_A(x)v+(-1)^{|x|}fxdv=F_A(d_A(x))fv+(-1)^{|x|}F_A(x)fdv\\
	=&pF_A(d_A(x))+(-1)^{|x|}F_A(x)d=d_A(F_A(x))+(-1)^{|F_A(x)|}F_A(x)d=dF_A(x)
	\end{align*}
	Thus $x'=F_A(x)$ as claimed.
	
	For the last relation we have to again look into the detail of the Verschiebung on $\pi_*(\R(\S)\otimes A)$. The element $xf$ lives in $\pi_*( A)\subset\pi_*(p^*\S[S^1]\otimes A_0[V])$ and on this summand the Verschiebung is given as follows:
	\[
	\begin{tikzcd}
	p_!^n((p^*)^{n+1}\S[S^1]\otimes A_0)\ar[r]\ar[d,phantom,sloped,"\simeq"]& p^*p_!^{n+1}((p^*)^{n+j}\S[S^1]\otimes A_0)\ar[d,phantom,sloped,"\simeq"]\\
	p^*\S[S^1]\otimes p_!^nA_0\ar[r]&p^*\S[S^1]\otimes p^*p_!^{n+1}A_0
	\end{tikzcd}
	\]
	which is precisely the Verschiebung on $ A_0[V]$, as described in \eqref{free_Verschiebung}, tensored with $p^*\S[S^1]$ and the last claim follows.
\end{proof}

Let us denote by $I_\R$ the ideal in $\pi_*(\S[V])\{v,f,d\}$ generated by $I_{\TCart}$ together with the relation from Lemma \ref{relations}. Then in the ring $\pi_*(A)\{v,f,d\}/I_\R$ every string of symbols $v, f,d$ and $x\in\pi_*(A)$ can be written as a multiple of one of the following generators:
\begin{align*}
v^i,\quad dv^i\qquad i\geq 0\\
f^j,\quad f^j d \qquad j>0
\end{align*}
As a $\pi_*(A)$-module this is precisely equivalent to what we computed $\pi_*\R_A\cong\pi_*(\R(\S)\otimes A)$ to be additively. In particular, we have now proven the case $A=A_0[V]$ of the following theorem:

\begin{theorem}\label{main_theorem} Given an algebra $A\in\Alg(\TCart_p)$, the $\mathbb{E}_1$-algebra structure on $\R_A$ induces an isomorphism of associative graded rings
	\[\pi_*\R_A\cong(\pi_*A)\{v,f,d\}/I_K\]
	where $d$ sits in degree 1, $v,f$ in degree 0 and the ideal $I_K$ is generated by the relations:
	\begin{alignat*}{3}
	f v&=p & vxf&=V_A(x)\\
	d f&=pf  d&v  d&=p d v\\
	fx&=F_A(x)f\qquad&xv&=F_A(x)v\\
	f  d v&= d\ (+\eta\ \text{if}\ p=2)& d^2&=\eta  d\\
	dx&=d_A(x)+(-1)^{|x|}xd
	\end{alignat*}
	for all $x\in\pi_*A$. Here $V_A,F_A\colon\pi_*A\to\pi_*A$ come from the Verschiebung and Frobenius and $d_A\colon\pi_*A\to\pi_{*+1}A$ from the $S^1$-action on the topological Cartier module $A$.
\end{theorem}
\begin{proof}
	Let $A\in\Alg(\TCart_p)$, then the counit of the free-forget adjunction between $\TCart_p$ and $\CycSp_p^\Fr$ gives a map $A[V]\to A$ of topological Cartier modules. The explicit identification of the Verschiebung and Frobenius structure on $\R(\S)\boxtimes A[V]$ above descends along the induced map
	\[\R(\S)\boxtimes A[V]\to\R(\S)\boxtimes A\]
	and the generators $v^i, dv^i, f^j$ and $f^jd$ with $i\geq0, j>0$ we constructed on $\R(\S)\boxtimes A[V]$ give here
	\[\pi_*(R(\S)\boxtimes A)\simeq\bigoplus_{i\geq 0}\left(v^i\cdot\pi_*A\oplus dv^i\cdot \pi_*A\right)\oplus\bigoplus_{j> 0}\left(\pi_*A\cdot f^j\oplus \pi_*A\cdot f^jd\right).\]
	Moreover, because the map $A[V]\to A$ preserves the topological Cartier module structure, it is compatible with the $V,F$ and $d$ operator on the homotopy groups. And because on underlying spectra it splits, thus is surjective on $\pi_*$, we recover all relations claimed. 

	Finally, this describes the associative ring structure on the endomorphism spectrum of $\R(\S)\boxtimes A$ and we get get the claimed description of $\pi_*\R_A$.
\end{proof}

\printbibliography
\end{document}